\documentclass[12pt]{amsart}

\usepackage{amsmath,amssymb,amsfonts,amscd,pictexwd,dcpic}

\newtheorem{theorem}{Theorem}[section]
\newtheorem{lemma}[theorem]{Lemma}
\newtheorem{proposition}[theorem]{Proposition}
\newtheorem{definition}[theorem]{Definition}

\newcommand{\from}{\colon}

\begin{document}
\title[Cofinite Connectedness and Cofinite Group Actions]{Cofinite Connectedness and Cofinite Group Actions}

\author{Amrita Acharyya}
\address{Department of Mathematics and Statistics\\
University of Toledo, Main Campus\\
Toledo, OH 43606-3390}
\email{Amrita.Acharyya@utoledo.edu}

\author{Jon M. Corson}
\address{Department of Mathematics\\
University of Alabama\\
Tuscaloosa, AL 35487-0350}
\email{jcorson@ua.edu}

\author{Bikash Das}
\address{Department of Mathematics\\
University of North Georgia, Gainesville Campus\\
Oakwood, Ga. 30566}
\email{Bikash.Das@ung.edu}



\begin{abstract}
We have defined and established a theory of cofinite connectedness of a cofinite graph. Many of the properties of connectedness of topological spaces have analogs for cofinite connectedness. We have seen that if $G$ is a cofinite group and  $\Gamma=\Gamma(G,X)$ is the Cayley graph. Then $\Gamma$ can be given a suitable cofinite uniform topological structure so that $X$ generates $G$, topologically iff $\Gamma$ is cofinitely connected. 

Our immediate next concern is developing group actions on cofinite graphs. Defining the action of an abstract group over a cofinite graph in the most natural way we are able to characterize a unique way of uniformizing an abstract group with a cofinite structure, obtained from the cofinite structure of the graph in the underlying action, so that the afore said action becomes uniformly continuous. 
\end{abstract}

\maketitle


\section{Introduction} \label{s:Intro}

A cofinite graph $\Gamma$ is said to be {\it cofinitely connected\/} if for each compatible cofinite equivalence relation $R$ on $\Gamma$, the quotient graph $\Gamma/R$ is path connected. 

Similar to the standard connectedness arguements for finite graphs or general topological spaces we were able to establish that the following statements are equivalent for any cofinite graph $\Gamma$:
\begin{enumerate}
\item $\Gamma$ is cofinitely connected;
\item $\Gamma$ is not the union of two disjoint nonempty subgraphs. 
\end{enumerate}
As an immediate consequence we obtained the following generalized characterization of connected Cayley graphs of cofinite groups:

Let $G$ be a cofinite group and let $\Gamma=\Gamma(G,X)$ be the Cayley graph. Then $\Gamma$ can be given a suitable cofinite topological graph structure so that $X$ generates $G$ $($topologically$)$  iff $\Gamma$ is cofinitely connected. 

Our final section is concerned with cofinite group actions on cofinite graphs. 

A group $G$ is said to act uniformly equicontinuously over a cofinite graph $\Gamma$ if and only if for each entourage $W$ over $\Gamma$ there exists an entourage $V$ over $\Gamma$ such that for all $g$ in $G, (g\times g)[V]\subseteq W$. In this case the group action induces a (Hausdorff) cofinite uniformity over $G$ if and only if the aforesaid action is faithful. 

We say that a group $G$ acts on a cofinite graph $\Gamma$ residually freely, if there exists a fundamental system of $G$-invariant compatible cofinite entourages $R$ over $\Gamma$ such that the induced group action of $G/N_R$ over $\Gamma/R$ is a free action, where $N_R$ is the Kernel of the action of $G$ on $\Gamma/R$.

Suppose that $G$ is a group acting faithfully and uniformly equicontinuously on a cofinite graph $\Gamma$, then the action $G\times\Gamma\to\Gamma$ is uniformly continuous. Also in that case $\widehat{G}$ acts on $\widehat{\Gamma}$ uniformly equicontinuously.


\section{Connected Cofinite Graphs}
A {\it path\/} in a graph $\Gamma$ is a finite string of edges $p=e_1\cdots e_n\in E(\Gamma)^*$ such that $t(e_i)=s(e_{i+1})$ for $1\le i\le n-1$. The {\it source\/} and {\it target\/} of this path $p$ are the vertices $s(p)=s(e_1)$ and $t(p)=t(e_n)$. We say that $\Gamma$ is {\it path connected\/} if there is a path in $\Gamma$ joining any two vertices. 

\begin{definition}\rm
A cofinite graph $\Gamma$ is {\it cofinitely connected\/} if for each compatible cofinite equivalence relation $R$ on $\Gamma$, the quotient graph $\Gamma/R$ is path connected. 
\end{definition}

\begin{proposition}
The following statements are equivalent for any cofinite graph $\Gamma$:
\begin{enumerate}
\item $\Gamma$ is cofinitely connected;
\item $\Gamma$ is not the uniform sum of two disjoint nonempty subgraphs. We then note that if $\Gamma$ is a profinite graph then we can restate the condition as $\Gamma$ is not the disjoint union of two nonempty closed subgraphs.  
\item the completion $\overline\Gamma$ of $\Gamma$ is cofinitely connected.
\end{enumerate}
\end{proposition}
\begin{proof}
(1) $\Rightarrow$ (2): If possible, let us assume that $\Gamma$ is the uniform sum of two disjoint subgraphs $\Gamma_1$ and $\Gamma_2$. Let $R_{\Gamma_1}$ be a compatible cofinite entourage over $\Gamma_1$ and $S_{\Gamma_2}$ be another compatible cofinite entourage over $\Gamma_2$. Then $W = R_{\Gamma_1}\cup S_{\Gamma_2}$ is a compatible cofinite entourage over $\Gamma$. 
Moreover $\Gamma/W$ is not path connected, a contradiction.

(2) $\Rightarrow$ (3): If possible, let us assume that $\overline{\Gamma}$ is not cofinitely connected. Hence there exists a compatible cofinite entourage $W$  over $\overline{\Gamma}$ such that $\overline{\Gamma}/W$ is not path connected.  

Let $\Sigma$ be a path connected component of $\overline{\Gamma}/W$. Hence $\Sigma$ is a subgraph of $\overline{\Gamma}/W$ and thus $(\overline{\Gamma}/W)\setminus\Sigma$ is a subgraph of $\overline{\Gamma}/W$ as well. Let $\Gamma_1 = \varphi^{-1}(\Sigma)$ and $\Gamma_2 = \varphi^{-1}(\overline{\Gamma}\setminus \Sigma)$, where $\varphi\from\overline{\Gamma}\to\overline{\Gamma}/W$ is the canonical quotient map. Then $\Gamma_1, \Gamma_2$ are closed subgraphs of $\overline{\Gamma} 
$ such that $\overline{\Gamma}$ is equal to the disjoint union of two closed subgraphs of $\overline{\Gamma}$ and then  $\overline{\Gamma}$ is equal to the uniform sum of two disjoint subgraphs of $\overline{\Gamma}$, a contradiction.

(3) $\Rightarrow$ (1): If possible assume that $\Gamma$ is not cofinitely connected. Then there exists a cofinite entourage $R$ over $\Gamma$ such that $\Gamma/R$ is not path connected. But, $\overline R$ is a compatible cofinite entourage over $\overline{\Gamma}$ such that $\Gamma/R$ is graph isomorphic to $\overline{\Gamma}/\overline R$. Hence $\overline{\Gamma}/\overline R$ is not path connected as well, a contradiction.  
\end{proof}

Many of the properties of connectedness of topological spaces have analogs for cofinite connectedness. Next we list a few of them.

\begin{proposition} \label{p:properties of connectedness}
Let $\Gamma$ be a cofinite graph and let $\Sigma$ be a uniform subgraph. 
\begin{enumerate}
\item If $\Sigma$ is path connected, then it is also cofinitely connected.
\item If $\Sigma$ is cofinitely connected, then so is the cofinite subgraph $\overline\Sigma$.
\item If $\Sigma$ is cofinitely connected and $f\from\Gamma\to\Delta$ a uniformly  continuous map of graphs, then $f(\Sigma)$ is also cofinitely connected $($as a cofinite subgraph of $\Delta)$.
\end{enumerate}
\end{proposition}
\begin{proof} Note that $\Sigma$ is a also a cofinite graph
\begin{enumerate}
\item If $\Sigma$ is path connected then any quotient graph of $\Sigma$ is path connected as well and thus our claim follows.
\item We will first see that $\overline{\Sigma} = \overline{V(\Sigma)\cup E(\Sigma)} = \overline{V(\Sigma})\cup\overline{E(\Sigma)}$ and that equals $V(\overline{\Sigma})\cup E(\overline{\Sigma})$ and thus is a cofinite subgraph of $\Gamma$ as well. Now, if possible suppose $\overline{\Sigma} = \Sigma_1\coprod\Sigma_2$, where $\Sigma_1, \Sigma_2$ are two disjoint nonempty cofinite subgraphs of $\overline{\Sigma}$. Then $\Sigma_1\cap\Sigma, \Sigma_2\cap\Sigma$ are two disjoint connected cofinite subgraphs of $\Sigma$. Let $R_1, R_2$ be two compatible cofinite entourage over $\Sigma_1\cap\Sigma, \Sigma_2\cap\Sigma$ respectively. Then there exist two compatible cofinite entourages $\tilde{R_1}, \tilde{R_2}$ over $\Sigma_1, \Sigma_2$ respectively such that $R_1 \supseteq \tilde{R_1}\cap(\Sigma\times\Sigma)$ and $ R_2$ contains $\tilde{R_2}\cap(\Sigma\times\Sigma)$. But as $\tilde{R_1}\cup \tilde{R_2}$ is a compatible cofinite entourage over $\overline{\Sigma}$, then $(\tilde{R_1}\cup\tilde{R_2})\cap(\Sigma\times\Sigma)$ is equal to $\tilde{R_1}\cap(\Sigma\times\Sigma)\cup \tilde{R_2}\cap(\Sigma\times\Sigma)$ which is a subset of $R_1\cup R_2$. So $R_1\cup R_2$  is a compatible entourage over $\Sigma$. Hence $\Sigma = (\Sigma_1\cap \Sigma)\coprod(\Sigma_2\cap\Sigma)$. Now suppose $\Sigma_1\cap\Sigma = \emptyset$. Then $\Sigma\subseteq \Sigma_2$. However $\Sigma_2$ is closed in $\overline{\Sigma}$ and hence closed in $\Gamma$. Thus $\overline{\Sigma}\subseteq \Sigma_2$ and therefore $\Sigma_1 = \emptyset$, a contradiction. Thus $\overline{\Sigma}$ is cofinitely connected.
\item Let $S$ be a compatible cofinite entourage over $f(\Sigma)$. Then as $f|_{\Sigma}\from\Sigma\to f(\Sigma)$ is uniformly continuous there is a compatible cofinite entourage $R$ over $\Sigma$ such that $R\subseteq (f\times f)^{-1}[S]$. Let us define $g\from\Sigma/R\to f(\Sigma)/S$ via $g(R[a]) = S[f(a)]$, for all $a\in\Sigma$. Now if $R[a] = R[b]$, then $(a,b)\in R$. Hence $(f(a),f(b))$ is in $S$ which implies that $ S[f(a)] = S[f(b)]$. Therefore $g$ is well defined and as $f$ is a map of graphs and both of $\Sigma/R, f(\Sigma)/S$ are discrete, $g$ is a surjective uniformly continuous map of graphs. Since $\Sigma/R$ is path connected then so is $g(\Sigma/R) = f(\Sigma)/S$.
\end{enumerate}
\end{proof}


\section{Cofinite Groups and their Cayley Graphs}

\begin{definition}
Let $G$ be an abstract group and $X = \{*\}\stackrel{\cdot}{\cup}E(X)$ be an abstract graph such that there is a map of sets $\alpha\from X\to G$ with $\alpha(*) = 1_G$, $(\alpha(e))^{-1} = \alpha(\overline e)$, for all $e\in E(X)$. Then the Cayley Graph $\Gamma(G,X)$ is defined as follows:
\begin{enumerate}
\item $V(\Gamma(G,X)) = G\times \{*\}, E(\Gamma(G,X)) = G\times E(X)$. 
\item $s(g,e) = (g,*)$, $t(g,e) = (g\alpha(e),*)$, $\overline{(g,e)} = (g\alpha(e),\overline{e})$.
\end{enumerate}
\end{definition}

Thus it follows that
\begin{enumerate}
\item $\Gamma(G,X) = V(\Gamma(G,X)) \stackrel{\cdot}{\cup}  E(\Gamma(G,X))$.
\item $s, t, \overline{\phantom e}$ are well defined and $t(\overline{(g,e)}) = t(g\alpha(e),\overline{e}) = (g\alpha(e)\alpha(\overline{e}),*)$\\
$= (g\alpha(e)\alpha(e)^{-1}),*) = (g,*) = s(g,e)$; $s(\overline{(g,e)}) = s(g\alpha(e),\overline{e})$\\
$= (g\alpha(e),*) = t(g,e)$.
\item If possible, let $(g,e) = \overline{(g,e)} = ((g\alpha(e),\overline{e})$ and thus $e = \overline e$, a contradiction. Finally, $\overline{\overline{(g,e)}} = \overline{(g\alpha(e),\overline{e})}$\\
$= (g\alpha(e)\alpha(\overline{e}),\overline{\overline e}) = (g\alpha(e)\alpha(e)^{-1}),e) = (g,e)$. Hence $\Gamma(G,X)$ is indeed a graph. 
\end{enumerate}

We say that $\alpha\from X\to G$ generates $G$ algebraically if $\langle\alpha(X)\rangle = G$. Equivalently, $\alpha\from X\to G$ generates $G$ algebraically if the unique extension to $\alpha\from E(X)^*\to G$ is onto.

\begin{lemma}
The Cayley graph $\Gamma(G,X)$ is path connected if and only if $\alpha\from X\to G$ generates $G$ algebraically.
\end{lemma}

\begin{definition}
Let $G$ be a cofinite group and $X = \{*\}\stackrel{\cdot}{\cup}E(X)$ be a cofinite graph such that there is a uniform continuous map of spaces $\alpha\from X\to G$ with $\alpha(*) = 1_G$, $(\alpha(e))^{-1} = \alpha(\overline e)$, for all $e\in E(X)$. Then the cofinite Cayley Graph $\Gamma(G,X)$ is defined as follows:
\begin{enumerate}
\item $V(\Gamma(G,X)) = G\times \{*\}, E(\Gamma(G,X)) = G\times E(X)$. 
\item $s(g,e) = (g,*)$, $t(g,e) = (g\alpha(e),*)$, $\overline{(g,e)} = (g\alpha(e),\overline{e})$.
\end{enumerate} 
$\Gamma(G,X)$ is endowed with the product uniform topological structure obtained from $G\times X = G\times V(\Gamma(G,X))\stackrel{\cdot}{\cup}G\times E(\Gamma(G,X))$.
\end{definition}
We have already seen that $\Gamma(G,X)$ is an abstract graph. Also being the product of Hausdorff, cofinite spaces, $\Gamma(G,X)$ is a Hausdorff, cofinite space as well. So in order to check that $\Gamma(G,X)$ is a cofinite graph it remains to prove that the compatible cofinite entourages over $\Gamma(G,X)$ forms a fundamental system of entourages. So it suffices to show that the family of cofinite entourages of the form $R\times S$, where $R$ is a cofinite congruence over $G$ and $S$ is a compatible cofinite entourage over $X$ such that $(\alpha\times \alpha)[S]\subseteq R$ forms a fundamental system of entourages. 

To establish the above claim let us first see that the cofinite entourages of the form $R\times S$ are indeed compatible. 
\begin{enumerate}
\item Let $((x,y),(p,q))\in R\times S$. So $(x,p)\in R\subseteq G\times G$ and $(y,q)$ is in $S$. Thus either $(y,q)\in S_V$ or $(y,q)\in S_E$ which implies that $y = * =q$ or $(y,q)\in S_E$. Hence $(x,y), (p,q)\in V(\Gamma(G,X))$ or $(x,y), (p,q)\in E(\Gamma(G,X))$. Hence $R\times S\subseteq (R\times S)_V\stackrel{\cdot}{\cup}(R\times S)_E$. The other direction of the inclusion follows more immediately.
\item Let $((g_1,e_1),(g_2,e_2))\in R\times S$. Then $(g_1,g_2)\in R$ and $(e_1,e_2)$ is in $S$. This implies that $(\alpha\times \alpha)(e_1,e_2) = (\alpha(e_1),\alpha(e_2))\in R$ and $(\overline{e_1},\overline{e_2})\in S$. Hence $(g_1\alpha(e_1)$, $g_2\alpha(e_2))\in R$, which implies $((g_1,*), (g_2,*))$ and $((g_1\alpha(e_1),*), (g_2\alpha(e_2),*))$ as well as $((g_1\alpha(e_1),\overline{e_1}), (g_2\alpha(e_2)\overline{e_2}))$ is in $R\times S$. Hence $(s(g_1,e_1), s(g_2,e_2))$, $(t(g_1,e_1), t(g_2,e_2))$ and $(\overline{(g_1,e_1)}$, $\overline{(g_2,e_2)})\in R\times S$.
\item If possible let $(\overline{(g_1,e_1)},(g_1,e_1))\in R\times S$ so $((g_1\alpha(e_1),\overline{e_1}),(g_1,e_1))$ is in $R\times S$. Thus $(\overline{e_1},e_1)\in S$, a contradiction.
\end{enumerate}
Now let $R\times T$ be any cofinite entourage over $G\times X$. Note that since $\alpha$ is uniformly continuous and $R$ is a cofinite congruence  over $G$, $T$ is a cofinite entourage over $X$, $(\alpha\times \alpha)^{-1}[R]\cap T$ is a cofinite entourage over $X$ and $(\alpha\times \alpha)[(\alpha\times \alpha)^{-1}[R]\cap T]\subseteq R$. So in particular one can take $S$ to be a compatible cofinite entourage over $X$ such that $S\subseteq (\alpha\times \alpha)^{-1}[R]\cap T$.  Then $(\alpha\times \alpha)[S]\subseteq R$ and $R\times S\subseteq R\times T$. This proves that $\Gamma(G,X)$ is a cofinite graph. We say that $\alpha\from X\to G$ generates $G$ topologically if $\overline{\langle\alpha(X)\rangle} = G$. 

\begin{theorem}
Let $\Gamma=\Gamma(G,X)$ be the cofinite Cayley graph. $\alpha$ from $X$ to $G$ generates $G$ topologically iff $\Gamma$ is cofinitely connected.
\end{theorem}
\begin{proof}
Let us first assume that $\alpha\from X\to G$ topologically generates $G$ and let $T$ be a compatible cofinite entourage over $\Gamma$, say $T$ is equal to $R\times S$ where $R$ is a cofinite congruence over $G$ and $S$ is a compatible cofinite entourage over $X$ where $S\subseteq (\alpha\times\alpha)^{-1}[R]$. Let us define $\alpha_{RS}\from X/S\to G/R$ via $\alpha_{RS}(S[x]) = R[\alpha(x)]$. Clearly, $\alpha_{RS}$ is well defined and $\alpha_{RS}(S[*]) = R[1_G]$ and $\alpha_{RS}(\overline{S[e]}) = R[\alpha(\overline e)] = R[(\alpha(e))^{-1}]$\\
$= (R[\alpha(e)])^{-1} = (\alpha_{RS}(S[e]))^{-1}$, for all $S[e]\in E(X/S)$. Let us now see that $\Gamma/T\cong\Gamma(G/R,X/S)$. Define $\theta\from\Gamma/T\to\Gamma(G/R,X/S)$ via $\theta(T[(g,x)]) = (R[g],S[x])$ for all $x$ in $X$ and all $g$ in $G$. Clearly, it is well defined injection as $T[(h,y)] = T[(g,x)]$ if and only if $((h,y),(g,x))\in T$ if and only if $(h,g)\in R, (y,x)\in S$ if and only if $R[h] = R[g]$ and $S[x] = S[y]$ if and only if $(R[h],S[y]) = (R[g],S[x])$. Also for all $(R[g],S[x])\in\Gamma(G/R,X/S)$, there exists $T[(g,x)]\in \Gamma/T$ such that $\theta(T[(g,x)]) = (R[g],S[x])$. Moreover it can easily be seen that $\theta$ is a map of graphs as $\theta(T[(g,*)])$ which is equal to $(R[g],S[*])$ belongs to $V(\Gamma(G/R,X/S))$ and $\theta(T[(g,e)])$ which equals to $(R[g],S[e])$ belongs to $E(\Gamma(G/R,X/S))$. Further more for all $(T[(g,e)])$ in $E(\Gamma/T)$ we see that $\theta(s(T[(g,e)])) = \theta(T[s(g,e)])$ which also equals to $\theta(T[(g,*)])$ equal to $(R[g],S[*])$ and that equals to $s(R[g],S[e])$. We also notice that $\theta(t(T[(g,e)])) = \theta(T[t(g,e)]) = \theta(T[(g\alpha(e),*)]) = (R[g\alpha(e)],S[*])$ which we know is equal to $(R[g]R[\alpha(e)],S[*]) = (R[g]\alpha_{RS}(S[e]),S[*])$ and that is equal to $t(R[g],S[e])$ Finally, $\theta(\overline{(T[(g,e)]}) = \theta(T[\overline{(g,e)}])$ and that equals $\theta(T[(g\alpha(e),\overline e)]) = (R[g\alpha(e)],S[\overline e])$ which can be written as $(R[g]R[\alpha(e)],S[\overline e]) = (R[g]\alpha_{RS}(S[e]),S[\overline e]) = (R[g]\alpha_{RS}(S[e]), \overline{S[e]})$ and that equals $\overline{(R[g],S[e])}$. Since $\Gamma/T, \Gamma(G/R, X/S)$ are discrete cofinite graphs, our claim follows. 

Now we wish to prove that $\langle\alpha_{RS}(X/S)\rangle = G/R$. Let $R[g]\in G/R$. Then as $\overline{\langle\alpha(X\rangle)} = G$, we have $R[g]\cap\langle\alpha(X)\rangle\neq\emptyset$. Let $a\in R[g]\cap \langle\alpha(X)\rangle$. So, $ R[g] = R[a]$. Also, since $a\in \langle\alpha(X)\rangle$, $a = \alpha(e_1)\alpha(e_2)\cdots \alpha(e_n)$, for some $e_1, e_2, \cdots, e_n\in E(X)$. Hence $R[a] = R[\alpha(e_1)]R[\alpha(e_2)]\cdots R[\alpha(e_n)]$, and one can represent this as $\alpha_{RS}(S[e_1])\alpha_{RS}(S[e_1])\cdots \alpha_{RS}(S[e_1])$. Thus $R[g] = R[a]\in\langle\alpha_{RS}(X/S)\rangle$. Therefore $\langle\alpha_{RS}(X/S)\rangle = G/R$ and consequently, $\Gamma/T = \Gamma(G/R,X/S)$ is path connected. Hence $\Gamma$ is cofinitely connected.  

Conversely, let us now take $\Gamma$ to be cofinitely connected. We want to show that $\overline{\langle\alpha(X)\rangle} = G$. So we intend to show that for any $g$ in $G$ and any open set $R[g]$ on $G, R[g]\cap\langle\alpha(X)\rangle\ne\emptyset$. We can form a compatible cofinite entourage $T = R\times S$ where $S$ is a compatible cofinite entourage over $X$ and $S\subseteq (\alpha\times\alpha)^{-1}[R]$. As earlier we can form the Cayley graph $\Gamma/T = \Gamma(G/R,X/S)$ and as $\Gamma$ is cofinitely connected, $\Gamma/T$ and therefore $\Gamma(G/R, X/S)$, is path connected. This implies $\langle\alpha_{RS}(X/S)\rangle = G/R$. So there is $e_1, e_2,\cdots,e_n$ in $E(X)$ such that $\alpha_{RS}(S[e_1])\alpha_{RS}(S[e_2])\cdots\alpha_{RS}(S[e_n]) = R[g]$. Thus we can finally say that $\alpha(e_1)\alpha(e_2)\cdots\alpha(e_n)\in R[g]$ which means $\langle\alpha(X)\rangle\cap R[g]\ne\emptyset$ and thus $\overline{\langle\alpha(X)\rangle} = G$. Hence $\alpha\from X\to G$ topologically generates $G$.
\end{proof}

\section{Groups Acting on Cofinite Graphs}\label{Group Action}

Let $G$ be a group and $\Gamma$ be a cofinite graph. We say that the group $G$ acts over $\Gamma$ if and only if 
\begin{enumerate}
\item For all $x$ in $\Gamma$, for all $g$ in $G, g.x$ is in $\Gamma$

\item For all $x$ in $\Gamma$, for all $g_1, g_2$ in $G,  g_1.( g_2.x)=( g_1 g_2).x$

\item For all $x$ in $\Gamma$, $1.x=x$

\item For all $v$ in $V(\Gamma)$, for all $g$ in $G, g.v$ is in $V(\Gamma)$ and for all $e$ in $E(\Gamma)$, for all $g$ in $G, g.e$ is in $E(\Gamma)$.

\item For all $e$ in $E(\Gamma)$, for all $g$ in $G, g.s(e)=s(g.e), g.t(e)=t(ge), g.(\overline{e})=\overline{g.e}$ 

\item There exists a $G-$invariant orientation $E^+(\Gamma)$ of $\Gamma$.
\end{enumerate}

Note that the aforesaid group action restricted to a singleton group element $g\in G$ can be treated as a well defined map of graphs, $\Gamma\to \Gamma$ taking $x\mapsto g.x$.  

\begin{definition}
A group $G$ is said to act uniformly equicontinuously over a cofinite graph $\Gamma$, if and only if for each entourage $W$ over $\Gamma$ there exists an entourage $V$ over $\Gamma$ such that for all $g$ in $G, (g\times g)[V]$ is a subset of $W$.
\end{definition}
\begin{lemma}
If $G$ acts uniformly equicontinuously over a cofinite graph $\Gamma$, then there exists a fundamental system of entourages consisting of $G$-invariant compatible cofinite entourages over $\Gamma$, i.e. for any entourage $U$ over $\Gamma$ there exists a compatible cofinite entourage $R$ over $\Gamma$ such that for all $g\in G, (g\times g)[R]\subseteq R\subseteq U$. 
\end{lemma}
\begin{proof}
 Let $U$ be any cofinite entourage over $\Gamma$. Then as $G$ acts uniformly equi continuously over $\Gamma$, there exists a compatible cofinite entourage $S$ over $\Gamma$ such that forall $g\in G,  (g\times g)[S]\subseteq U$. Choose a $G$-invariant orientation $E^+(\Gamma)$ of $\Gamma$. Without loss of generality, we can assume that our compatible equivalence relation $S$ on $\Gamma$ is {\it orientation preserving\/} i.e. whenever $(e,e')\in R$ and $e\in E^+(\Gamma)$, then also $e'\in E^+(\Gamma)$.  Clearly, $S\subseteq\cup_{g\in G}(g\times g)[S]\subseteq U$. Now if $S_0=\cup_{g\in G}(g\times g)[S]$ and $T=\langle S_0\rangle$, note that  $S\subseteq T\subseteq U$. Since for all $h\in G, (h\times h)[S_0]=S_0$ and $S_0^{-1} = S_0$ it follows that $T$ is in the transitive closure of $S_0$. Let $(x,y)\in  T$. Then there exists a finite sequence $x_0,x_1,..,x_n$ such that $(x_i,x_{i+1})\in S_0$, for all $i=0,1,2,...,n-1$ and $x = x_0, y = x_n$. Hence $(gx_i,gx_{i+1})\in S_0$, for all $i=0,1,2,...,n-1$, for all $g\in G$. Thus $(gx_0,gx_n)=(gx,gy)\in T$, for all $g\in G$. Hence for all $g\in G, (g\times g)[T]\subseteq T$ and our claim that $T$ is a $G$-invariant cofinite entourage, follows. It remains to check that $T$ is compatible. Let $(x,y)\in T$. If $(x,y)\in S_0$, then there is $(t,s)\in S = S_V\cup S_E$ and $g\in G$ such that $(gt,gs) = (x,y)$. Without loss of generality let $(t,s)\in S_V$. Then $(t,s)\in V(\Gamma)\times V(\Gamma)$ which implies that $ (x,y)\in T_V$. Now let $(x,y)\in  T\setminus S_0$. Then there exists a finite sequence $x_0,x_1,..,x_n$ such that $(x_i,x_{i+1})\in S_0$, for all $i=0,1,2,...,n-1$ and $x = x_0, y = x_n$. Hence by the previous argument if $(x_0,x_1)\in T_V$ then $(x_i,x_{i+1})\in T_V$, for all $i=1,2,...,n-1$. Thus $(x,y)\in T_V$. If $(x_0,x_1)\in T_E$ then $(x_i,x_{i+1})\in T_E$, for all $i=1,2,...,n-1$, which implies $ (x,y)\in T_E$. Let $(e_1,e_2)\in T$. If $(x,y)\in S_0$, then there is $(p,q)\in S$ and $g\in G$ such that $(gp,gq) = (e_1,e_2)$. Then $(s(p),s(q))\in S$. So $(s(e_1),s(e_2))$ which equals to $(gs(p),gs(q))$ is in $(g\times g)[S]\subseteq S_0$ so that $ (s(e_1),s(e_2))\in T$. Now let $(e_1,e_2)\in  T\setminus S_0$. Then there exists a finite sequence $x_0,x_1,..,x_n$ such that $(x_i,x_{i+1})\in S_0, \forall i=0,1,2,...,n-1$ and $e_1 = x_0, e_2 = x_n$. Hence by the previous argument $(s(x_i),s(x_{i+1}))\in T, \forall i=0, 1,2,...,n-1$ and thus   $(s(e_1),s(e_2))\in T$. Similarly, $(t(e_1),t(e_2))\in T$ and $(\overline{e_1},\overline{e_2})\in T$.  Finally, to show that for any $e\in E^+(\Gamma), (e\overline e)\in T$ it suffices to note that $T$ is orientation preserving. Alternatively, if possible let $(e,\overline e)\in T$. If $(e,\overline e)\in S_0$, then there is $(p,q)\in S$ and $g\in G$ such that $(gp,gq) = (e,\overline e)$. Then $\overline e = \overline{gp} = g\overline p = gq$ which implies that $ \overline p = q$, so $ (p,\overline p)\in S$, a contradiction. Now let $(e,\overline e)\in  T\setminus S_0$. Then there exists a finite sequence $x_0,x_1,..,x_n$ such that $(x_i,x_{i+1})\in S_0$, for all $i=0,1,2,...,n-1$ and $e = x_0, \overline e = x_n$.  Now let there is $(p,q)\in S$ and $g\in G$ such that $(gp,gq) = (x_0,x_1)$. Without loss of generality we may assume $(p,q)\in E^+(\Gamma)\times E^+(\Gamma)$. Then $(gp,gq) = (x_0,x_1)\in E^+(\Gamma)\times E^+(\Gamma)$.  Hence $(x_i,x_{i+1})\in E^+(\Gamma)\times E^+(\Gamma)$, for all $i=1,2,...,n-1$ which implies that $ (e,\overline e)\in E^+(\Gamma)\times E^+(\Gamma)$, a contradiction. Our claim follows. 
\end{proof}
\begin{definition}
We say a group $G$ acts on a cofinite space $\Gamma$ faithfully, if for all $g$ in $G\setminus\{1\}$ there exists $x$ in $\Gamma$ such that $gx$ is not equal to $x$ in $\Gamma$.
\end{definition}
\begin{lemma}\label{equi continuous}
Let $G$ acts on a cofinite graph $\Gamma$ uniformly equicontinuously. Then $G$ acts on $\Gamma/R$ and  $G/N_R$ acts on $\Gamma/R$ as well, where $R$ is a $G$-invariant compatible cofinite entourage over $\Gamma$. If $\{R\mid R\in I\}$ is a fundamental system of $G$-invariant compatible cofinite entourages over $\Gamma$, then $\{N_R\mid R\in I\}$ forms a fundamental system of cofinite congruences for some uniformity over $G$. 
\end{lemma}
\begin{proof}
Let $R$ be a $G$-invariant compatible cofinite entourage over $\Gamma$. Let us define $G\times\Gamma/R\to\Gamma/R$ via $g.R[x]=R[g.x]$, for all $g\in G$, for all $x\in \Gamma$. Now let $R[x]=R[y]$ so $(x,y)\in R$ which implies that $(g.x,g.y)\in R$. Then $R[g.x]=R[g.y]$. Hence the induced group action is well defined. 

Let us now consider the group action $G/N_R\times \Gamma/R\to \Gamma/R,$ defined via $N_R[g].R[x]=R[g.x]$,for all $x\in \Gamma$, for all $g\in G$. Now let $(N_R[g],R[x])=(N_R[h],R[y])$ which implies that $ (g,h)\in N_R, (x,y)$ is in $R$. Then $(g.x,h.x)\in R$, as $h^{-1}\in G, (h^{-1}g.x,h^{-1}h.x)\in R$. So $(h^{-1}g.x,y)\in R$. Thus $(g.x,h.y)\in R$ which implies that $R[g.x]$ equals to $R[h.y]$. Hence the induced group action is well defined. Let us now show that $N_R$ is an equivalence relation over $G$, for all $G$-invariant compatible cofinite entourage $R$ over $\Gamma$. 

\begin{enumerate}
\item for all $g\in G$, for all $x\in \Gamma, (g.x,g.x)\in R$. Hence $(g,g)\in N_{R}$, for all $g\in G$ which implies that $D(G)\subseteq N_R$.
\item Now $(h,g)\in N_R^{-1}\Leftrightarrow (g,h)\in N_R\Leftrightarrow (g.x,h.x)\in R$, for all $x\in \Gamma$. Thus $(g.x,h.x)\in R\Leftrightarrow (h.x,g.x)\in R$, for all $x\in \Gamma$. Hence $(h.x,g.x)\in R\Leftrightarrow (h,g)\in N_R$. Thus $N_R^{-1}=N_R$.
\item Let $(g,h), (h,k)\in N_R$. This implies $ (g.x,h.x),(h.x,k.x)$ is in $R, \forall x\in \Gamma$. Hence $(g.x,k.x)\in R$, for all $x\in \Gamma$. So $(g,k)\in N_R$ which implies that $(N_R)^2\subseteq N_R$.
\end{enumerate}
Also we now check that $N_R$ is a congruence over $G$. For, let us take $(g_1,g_2),(g_3,g_4)\in N_R$. Then for all $x\in \Gamma, (g_1.x,g_2.x),(g_3.x,g_4.x)\in R$; for all $x\in \Gamma, g_3.x\in \Gamma$ and so $(g_1g_3.x,g_2g_3.x)\in R$ and $(g_2g_3.x,g_2g_4.x)$ is in $R$, since $R$ is $G$-invariant. Thus $(g_1g_3.x,g_2g_4.x)\in R$, for all $x\in \Gamma$ so that $(g_1g_3,g_2g_4)\in N_R$. Thus our claim follows. Let us now show that $G/N_R$ is finite. Furthermore, define  $g\from\Gamma/R\to\Gamma/R$ as $g$ maps $(R[x])$ into $R[g.x]$. Now, $R[x]=R[y]\Longleftrightarrow (x,y)\in R$ if and only if $(g.x,g.y)\in R\Longleftrightarrow R[g.x]=R[g.y]$. Hence the map $g$ is a well defined injection. Now for all $R[x]\in\Gamma/R$ there exists $g^{-1}R[x]\in\Gamma/R$ such that $g(g^{-1}R[x])$ equals to $R[x]$. Hence $g\in Sym(\Gamma/R)$. Now let us define a map $\theta\from G/N_R\to Sym(\Gamma/R)$ via $\theta(N_R[g])=g.$ Now $N_R[g_1]$ equals to $N_R[g_2]$ if and only if $(g_1,g_2)\in N_R$ if and only if $(g_1.x,g_2.x)\in R$ for all $x\in\Gamma$. Hence $(g_1.x,g_2.x)\in R$ if and only if $R[g_1.x]=R[g_2.x]$ if and only if $g_1(R[x])=g_2(R[x])$ $g_1=g_2$ in $Sym(\Gamma/R)$. Hence $\theta$ is a well defined injection. Thus $\left|G/N_R\right|\leq\left|Sym(\Gamma/R)\right|<\infty$ as $\left|\Gamma/R\right|<\infty$. So, next we will like to show that $\{N_R\mid R\in I\}$ forms a fundamental system of cofinite congruences over $G$. 
\begin{enumerate}
\item $D(G)\subseteq N_R$, for all $R\in I$, as $N_R$ is reflexive.
\item Now for some $R,S\in I, (g_1, g_2)\in N_R\bigcap N_S$ if and only if $(g_1.x, g_2.x)\in R\bigcap S$, for all $x\in \Gamma\Leftrightarrow (g_1, g_2)\in N_{R\bigcap S}$. Thus $N_R\bigcap N_S=N_{R\bigcap S}$.
\item For all $N_R, N_R^2= N_R$, as $N_R$ is transitive.
\item For all $N_R, N_R^{-1}=N_R$, as $N_R$ is symmetric. 
\end{enumerate}
Hence our claim follows.
\end{proof}
\begin{definition}
We say that a group $G$ acts on a cofinite graph $\Gamma$ residually freely, if there exists a fundamental system of $G$-invariant compatible cofinite entourages $R$ over $\Gamma$ such that the induced group action of $G/N_R$ over $\Gamma/R$ is a free action.
\end{definition}
\begin{lemma}
$N_R[1]$ is a finite index normal subgroup of $G$ and $G/N_R[1]$ is isomorphic with $G/N_R$. More generally, if $N$ is a congruence on $G$, then $N[1]$ is a normal subgroup of $G$ and $G/N[1]\cong G/N$. 
\end{lemma} 
\begin{proof}
Let us first see that $N_R[1]\triangleleft_fG$ for all $G$-invariant compatible cofinite entourage $R$ over $\Gamma$. Let $g,h\in N_R[1]$. This implies $(1,g)\in N_R$ and hence $(g,1), (1,h) \in N_R$. Thus $(g,h)\in N_R$. This implies $(g.x,h.x)$ is in $R$, for all $x\in \Gamma$ and so $(x,g^{-1}h.x)\in R$, for all $x\in \Gamma$. Hence, $(1,g^{-1}h)$ is in $N_R$ and thus $g^{-1}h\in N_R[1]$. So, $N_R[1]\leq G$. For all $g\in G$, for all $x\in \Gamma, g.x\in \Gamma$. Hence for all $k\in N_R[1], (x,k.x)\in R$, hence $(k.x,x)$ is in $R$. Thus $(kg.x,g.x)\in R$ and $(g^{-1}kg.x,g^{-1}g.x)=(g^{-1}kg.x,x)\in R$. Hence $(g^{-1}kg,1)\in N_R$. So, $g^{-1}kg\in N_R[1]$ and thus $N_R[1]\triangleleft G$. Now let us define $\eta$ from $G/N_R[1]$ to $G/N_R$ via $\eta(gN_R[1]) = N_R[g]$. Then, $gN_R[1]$ is equal to $hN_R[1]$ if and only if $h^{-1}g\in N_R[1]$ if and only if $(1,h^{-1}g)\in N_R$ if and only if $(x,h^{-1}g.x)\in R$ if and only if $(h.x,g.x)\in R$ if and only if $(h,g)\in N_R$ if and only if $N_R[h]=N_R[g]$, for all $x$ in $\Gamma$. Thus $\eta$ is a well defined injection and hence $\left|G/N_R[1]\right|\leq \left|G/N_R\right|<\infty$. Hence $N_R[1]\triangleleft_fG$. Let us check that $G/N_R$ is a group. For, let $N_R[g_i]$ is in $G/N_R, i=1,2$. Then $N_R[g_1]N_R[g_2]=N_R[g_1g_2]\in G/N_R$. Let $N_R[g_i]$ in $G/N_R$, for $i=1,2,3$. Then $(N_R[g_1]N_R[g_2])N_R[g_3]$ which is equal to $N_R[g_1g_2]N_R[g_3]$ and that equals to $N_R[g_1g_2g_3] = N_R[g_1]N_R[g_2g_3]$ which is equal to $N_R[g_1](N_R[g_2]N_R[g_3])$.  For all $N_R[g]\in G/N_R$, there exists $N_R[1]$ in $G/N_R$, such that $N_R[1] N_R[g] = N_R[g] = N_R[g] N_R[1]$. For all $N_R[g]$ in $G/N_R$, there exists $N_R[g^{-1}]$ in $G/N_R$, such that $N_R[g^{-1}]N_R[g]$ equals to $N_R[g^{-1}g] = N_R[1] = N_R[gg^{-1}] = N_R[g]N_R[g^{-1}]$. Hence our claim. Now let us define $\zeta\from G/N_R[1]\to G/N_R$ via $\zeta(gN_R[1])=N_R[g]$. Then for $g_1, g_2$  in $G, g_1N_R[1] = g_2N_R[1]$ if and only if $g_2^{-1}g_1\in N_R[1]$ if and only if $(1,g_2^{-1}g_1)\in N_R$ if and only if $(x,g_2^{-1}g_1.x)\in R$ if and only if $(g_2.x,g_1.x)\in R$ if and only if $(g_2,g_1)\in N_R$ if and only if $N_R[g_2]$ equals to $N_R[g_1]$. Hence $\zeta$ is a well defined injection. Also for all $N_R[g]$ in $G/N_R$, there exists $gN_R[1]\in G/N_R[1]$ such that $\zeta(gN_R[1]) = N_R[g]$. Thus $\zeta$ is surjective as well. Also for $g_1N_R[1], g_2N_R[1]\in G/N_R[1]$, we have $\zeta(g_1N_R[1]g_2N_R[1]) = \zeta(g_1g_2N_R[1])$ and that equals to $N_R[g_1g_2]$ which equals to $N_R[g_1]N_R[g_2] = \zeta(g_1N_R[1])\zeta(g_2N_R[1])$. Hence $\zeta$ is a group homomorphism and thus a group isomorphism. Also, both $G/N_R[1], G/N_R$, are finite discrete topological groups, so $\zeta$ is an isomorphism of uniform cofinite groups as well.
\end{proof}
\begin{lemma}
The induced uniform topology over $G$ as in Lemma~\ref{equi continuous} is Hausdorff if and only if $G$ acts faithfully over $\Gamma$.
\end{lemma}
\begin{proof}
Let us first assume that $G$ acts faithfully over $\Gamma$. Now let $g\neq h$ in $G$. Then $h^{-1}g\neq 1$. So there exists $x\in \Gamma$ such that $h^{-1}g.x\neq x$ implying that $ g.x\neq h.x$. Then there exists a $G$-invariant compatible cofinite entourage $R$ over $\Gamma$ such that $(g.x,h.x)\notin R$, as $\Gamma$ is Hausdorff. Hence $(g,h)\notin N_R$. Thus $G$ is Hausdorff. 

Conversely, let us assume that $G$  is Hausdorff and let $g\neq 1$ in $G$. Then there exists some $G$-invariant compatible cofinite entourage $R$ over $\Gamma$ such that $(1,g)\notin N_R$. Hence there exists $x\in\Gamma$ such that $(x,g.x)\notin R$. Hence $R[x]\neq R[g.x]$ so that $x\neq g.x$. Our claim follows.
\end{proof}
\begin{lemma}\label{uniform continuous group action} 
Suppose that $G$ is a group acting uniformly equicontinuously on a cofinite graph $\Gamma$ and give $G$ the induced uniformity as in Lemma~\ref{equi continuous}. Then the action $G\times\Gamma\to\Gamma$ is uniformly continuous.
\end{lemma}
\begin{proof}
Let $R$ be a $G$-invariant cofinite entourage over $\Gamma$. Now let $((g,x),(h,y))\in N_R\times R$, i.e. $(g,h)\in N_R, (x,y)\in R$. Now $x$ in $\Gamma$ and $(gx,hx)\in R$ this implies $(h^{-1}gx,x)\in R$. We have  $(h^{-1}gx,y)\in R$ and hence $(gx,hy)\in R$. Thus our claim.
\end{proof}

Now if $R\leq S$ in $I$, then $S\subseteq R$. Let $(g_1,g_2)\in N_S$. Then $(g_1x,g_2x)\in S$, for all $x\in\Gamma$ and hence $(g_1x,g_2x)\in R$,for all $x\in \Gamma $  which implies $ (g_1,g_2)\in N_R$. Thus $ N_S\subseteq N_R$. For all $R\leq S$, in $I$, let us define $\psi_{RS}\from G/N_S\to G/N_R$ via $\psi_{RS}(N_S[g])=N_R[g]$. Then $\psi_{RS}$ is a well defined uniformly continuous group isomorphism, as each of $G/N_R, G/N_S$ are finite discrete groups. If $R=S$, then $\psi_{RR}=id_{G/N_R}$. And if $R\leq S\leq T$, then $\psi_{RS}\psi_{ST}=\psi_{RT}$. Then $\{G/N_R\mid R\in I, \psi_{RS}, R\leq S\in I\}$, forms an inverse system of finite discrete groups. Let $\widehat{\Gamma}=\varprojlim_{R\in I}\Gamma/R$ and $\widehat{G}=\varprojlim_{R\in I}G/N_R$, where $\psi_R\from\widehat{G}\to G/N_R$ is the corresponding canonical projection map. Now if $I_{1}, I_{2}$ are two fundamental systems of  $G$-invariant cofinite entourages over $\Gamma$, clearly $I_{1}, I_{2}$ will form fundamental systems of cofinite congruences, for two induced uniformities, over $G$. Now let $N_{R_1}$ be a cofinite congruence over $G$ for some $R_1\in I_{1}$. Then there exists a $R_2$, cofinite entourage over $\Gamma$, such that  $R_2\in I_{2}$ and $R_2\subseteq R_1$. Hence $N_{R_2}\subseteq N_{R_1}$. Now let $N_{S_2}$ be a cofinite congruence over $G$ for some $S_2\in I_{2}$. Then there exists $S_1$, cofinite entourage over $\Gamma$, such that  $S_1\in I_{1}$ and $S_1\subseteq S_2$. Hence $N_{S_1}\subseteq N_{S_2}$. Thus any cofinite congruence corresponding to the directed set $I_{1}$ is a cofinite congruence corresponding to the directed set $I_{2}$ and vice versa.  Thus the two induced uniform structures over $G$ are equivalent and so the completion of $G$ with respect to the induced uniformity, from the cofinite graph $\Gamma$, is unique up to both algebraic and topological isomorphism.

\begin{theorem}
If $G$ acts on $\Gamma$, as in Lemma~\ref{equi continuous}, faithfully then $\widehat{G}$ acts on $\widehat{\Gamma}$ uniformly equicontinuously.
\end{theorem}
\begin{proof}

The group $G$ acts on $\Gamma$ uniformly equicontinuously. We fix a $G$-invariant orientation $E^+(\Gamma)$ of $\Gamma$. By Lemma~\ref{uniform continuous group action} the action is uniformly continuous as well. Let $\chi\from G\times\Gamma\to \Gamma$ be this group action. Now since $\Gamma$ is topologically embedded in $\widehat{\Gamma}$ by the inclusion map, say, $i$, the map $i\circ\chi\from G\times\Gamma\to\widehat{\Gamma}$ is a uniformly continuous. Then there exists a unique uniformly continuous map $\widehat{\chi}\from\widehat{G}\times\widehat{\Gamma}\to\widehat{\Gamma}$ that extends $\chi$. We claim that $\widehat{\chi}$ is the required group action. We can take $\widehat{\Gamma} = \varprojlim\Gamma/R$ and $\widehat{G} = \varprojlim G/N_R$, where $R$ runs throughout all $G$-invariant compatible cofinite entourages of $\Gamma$ that are orientation preserving. Then $\widehat{G}\times \widehat{\Gamma} = \varprojlim (G/N_R\times \Gamma/R)$ and $G\times\Gamma$ is defined coordinatewise via $(N_R[g_R])_R.(R[x_R])_R=(R[g_R.x_R])_R$.  If possible  let, $((N_R[g_R])_R,(R[x_R])_R) = ((N_R[h_R])_R,(R[y_R])_R)$. So, $N_R[g_R]$ equals to $N_R[h_R]$ and $R[x_R]=R[y_R], \forall R\in I, (g_R,h_R)\in N_R$ and $(x_R,y_R) \in R$. This implies that $(g_R.x_R,h_R.x_R)\in R$ which further ensures that $(h_R^{-1}g_R.x_R,x_R)\in R$. Then $(h_R^{-1}g_R.x_R,y_R)\in R$ and $(g_R.x_R,h_R.y_R)\in R$. Hence $(R[g_R.x_R])_R = (R[h_R.y_R])_R$. So, the action is well defined. Let $g=(N_R[g_R])_R$ and $h=(N_R[h_R])_R$ in $\widehat{G}$, $x=(R[x_R])_R\in \widehat{\Gamma}$. Now $h.(g.x) = h.(R[g_R.x_R])_R = (R[h_Rg_R.x_R])_R$ which then equals to$(N_R[h_Rg_R])_R.x = (hg).x$. Hence the action is associative. Now $(N_R[1])_R.(R[x_R])_R = (R[1x_R])_R=(R[x_R])_R$. Furthermore for all $v$ equal to $(R[v_R])_R\in V(\widehat{\Gamma})$ and for all $g$ equal to $(N_R[g_R])_R\in \widehat{G}$ one can say that $g.v = (R[g_R.v_R])_R\in V(\widehat{\Gamma})$ as each $g_R.v_R\in V(\Gamma)$. Similarly, for all $e$ equal to $(R[e_R])_R$ in $E(\widehat{\Gamma})$ and for all $g$ equal to $(N_R[g_R])_R$ in $\widehat{G}$, $g.e = (R[g_Re_R])_R$ in $E(\widehat{\Gamma})$. For all $e$ equal to $(R[e_R])_R$ in $E(\widehat{\Gamma})$, for all $g$ equal to $(N_R[g_R])_R$ in $\widehat{G}$, we have $s(g.e) = s((R[g_Re_R])_R)$ and so $(R[g_Rs(e_R)])_R$ equals to $(g.(R[s(e_R)])_R$ and that equals to $g.s(e)$. Hence the properties $t(g.e)=g.t(e)$ and $\overline{g.e}=g.\overline{e}$ follow similarly. Finally, let $E^+(\widehat{\Gamma})$ consists of all the edges $(R[e_R])_R$, where $e_R\in E^+(\Gamma)$. Since each $R$ is orientation preserving, it follows that $E^+(\widehat{\Gamma})$ is an orientation of $\widehat{\Gamma}$. Since $E^+(\Gamma)$ is $G$-invariant, we see that $E^+(\widehat{\Gamma})$ is $\widehat{\Gamma}$-invariant. Hence this is a well defined group action. Also for all $g\in G$, and $x\in \Gamma$, $(N_R[g])_R.(R[x])_R$ equals to $(R[g.x])_R$ which equals to $g.x$ in $\Gamma$. Thus the restriction of this group action agrees with the group action $\chi$. Now $\{R\mid R\in I\}, \{N_R\mid R\in I\}$ is a fundamental system of cofinite entourages over $\Gamma$, is a fundamental system of cofinite congruences over $G$. Hence $\{\overline{R}\mid R\in I\}$ is a fundamental system of cofinite entourages over $\widehat{\Gamma}$ and $\{\overline{N_R}\mid R\in I\}$ is a fundamental system of cofinite congruences over $\widehat{G}$ respectively. Let us now see that the aforesaid group action is uniformly continuous. For let us consider the group action $G/N_R\times\Gamma/R\to\Gamma/R$ defined via $N_R[g]R[x]=R[g.x]$, which is uniformly continuous as both $G/N_R\times\Gamma/R$ and $\Gamma/R$ are finite discrete uniform topological spaces. Hence the group action, $\widehat{G}\times \widehat{\Gamma}\to\widehat{\Gamma}$ is uniformly continuous. Thus the aforesaid group action is our choice of $\widehat{\chi}$, by the uniqueness of $\widehat{\chi}$. So the restriction of the aforesaid action  $\{\widehat{g}\}\times\widehat{\Gamma}\to \widehat{\Gamma}$ is a uniformly continuous map of graphs, for all $\widehat{g}\in \widehat{G}$. We check that for all $(x,y)\in R$ and for all $\widehat{g}\in \widehat{G}$ the ordered pair $(\widehat{g}.x,\widehat{g}.y)\in \overline{R}$ . For, let $\widehat{g}=(N_R[g_R])_R\in \widehat{G}$ and for $x,y\in \Gamma,((R[x])_R, (R[y])_R)\in R$. Now $\overline{R}[(R[g_R.x])_R]=\overline{R}[g_R.x]$ becomes equal to $\overline{R}[g_R.y]=\overline{R}[(R[g_R.y])_R]$. So, $((N_R[g_R])_R(R[x])_R,(N_R[g_R])_R(R[y])_R)\in \overline{R}$. This implies $(\widehat{g}\times\widehat{g})[R]$ is a subset of $\overline{R}$. Thus for all $\widehat{g}\in \widehat{G}$ we observe that $(\widehat{g}\times\widehat{g})[\overline{R}]$ is a sub set of $\overline{\widehat{g}\times\widehat{g}[R]}$ which is a sub set of $\overline{\overline{R}} = \overline{R}$. Hence $\overline{R}$ is $\widehat{G}$ invariant.
\end{proof} 
Thus $\Phi_1=\{N_{\overline{R}}\mid R\in I\}$ and $\Phi_2=\{\overline{N_R}\mid R\in I\}$ form fundamental systems of cofinite congruences over $\widehat{G}$. Let $\tau_{\Phi_1}, \tau_{\Phi_2}$ be the topologies induced by $\Phi_1, \Phi_2$ respectively. 
\begin{theorem}
The uniformities on $\widehat{G}$ obtained by $\Phi_1$ and $\Phi_2$ are equivalent.
\end{theorem}
\begin{proof}
Let us first show that $N_{\overline{R}}\cap G\times G = N_R$. For, let $(g,h)\in N_R$. Then for all $x\in \Gamma, (g.x,h.x)\in R\subseteq \overline{R}$. Now let $(R[x_R])_R\in \widehat{\Gamma}$. Then $\overline{R}[g(R[x_R])_R]=\overline{R}[g.x_R]=\overline{R}[h.x_R]=\overline{R}[h(R[x_R])_R]$ which implies that $(g,h)\in N_{\overline{R}}\cap G\times G$. Thus, $N_R \subseteq  N_{\overline{R}}\cap G\times G$. Again, if $(g,h)$ belongs to $N_{\overline{R}}\cap G\times G$, then for all $x\in \Gamma\subseteq\widehat{\Gamma}$, and so $(g.x,h.x)\in\overline{R}\cap \Gamma\times \Gamma = R$ and this implies $(g,h)\in N_R$. Our claim follows. Then as uniform subgraphs $(G,\tau_{\Phi_1})\cong(G,\tau_{\Phi_2})$, both algebraically and topologically, their corresponding completions $(\widehat{G},\tau_{\Phi_1})\cong(\widehat{G},\tau_{\Phi_2})$, both algebraically and topologically. Since for all $S\in I$, $\psi_S\from G\to G/N_S$ is a uniform continuous group homomorphism and $G/N_S$ is discrete, there exists a unique uniform continuous extension of $\psi_S$, namely, $\widehat{\psi_S}\from\widehat{G}\to G/N_S$. Let us define $\lambda_S\from \widehat{G}\to G/N_S$ via $\lambda_S(g)=N_S[g_S]$, where $g=(N_R[g_R])_R$. Now let $g=(N_R[g_R])_R, h=(N_R[h_R])_R\in\widehat{G}$ be such that $g=h$ which implies that $ N_S[g_S]=N_S[h_S]$ and hence $\lambda_S$ is well defined. Now let $(g,h)\in N_{\overline{S}}$. First of all $N_{\overline{S}}[g_S]=N_{\overline{S}}[g]=N_{\overline{S}}[h]=N_{\overline{S}}[h_S]$. So, $(g_S,h_S)\in N_{\overline{S}}\bigcap G\times G=N_S$. Hence $N_S[g_S]=N_S[h_S]$ which implies that $ \lambda_S(g)=\lambda_S(h)$, so $(\lambda_S(g),\lambda_S(h))\in D(G/N_R)$. Thus $N_{\overline{S}}$ is a sub set of $(\lambda_S\times \lambda_S)^{-1}D(G/N_R)$. Hence $\lambda_S$ is uniformly continuous. Now for all $g, h\in\widehat{G}, \lambda_S(gh)=N_S[g_Sh_S]=N_S[g_S]N_S[h_S]=\lambda_S(g)\lambda_S(h)$ and for all $g\in G, \lambda_S(g)=\lambda_S((N_R[g])_R)=N_S[g]=\psi_S(g)$. Thus $\lambda_S$ is an well defined uniformly continuous group homomorphism that extends $\psi_S$. Then by the uniqueness of the extension, $\widehat{\psi_S}=\lambda_S$. Now $N_{\overline{S}}$ is a closed subspace of $\widehat{G}$, then $\overline{N_{\overline{S}}\cap G\times G} = \overline{N_S}$ which implies that $\overline{N_S}$ is a sub set of $\overline{N_{\overline{S}}}$ which equals to $N_{\overline{S}}$. Let us define $\theta$ from $\widehat{G}/N_{\overline{S}}$ to $G/N_S$ as $\theta$ takes $N_{\overline{S}}[g]$ into $N_S[g_S]$, where $g=(N_R[g_R])_R$. Now $N_{\overline{S}}[g]=N_{\overline{S}}[h]$ in $\widehat{G}/N_{\overline{S}}$ will imply $(g_S,h_S)$ is in $N_{\overline{S}}$ and this implies for all $x$ in $X$ the ordered pair $(g_Sx,h_Sx)$ is in $\overline{S}\bigcap\Gamma\times \Gamma$ which is eventually equal to $S$. Thus $(g_S,h_S)\in N_S$. Then $\theta(N_{\overline{S}}[g])$ equals to $N_S[g_S]$ which is equal to $N_S[h_S]$ and that equals  $\theta(N_{\overline{S}}[h])$. Hence $\theta$ is well defined. On the other hand let $N_{\overline{S}}[g]$, $N_{\overline{S}}[h]$ be such that $\theta(N_{\overline{S}}[g])$ equals $\theta(N_{\overline{S}}[h])$. Thus $N_S[g_S]$ equal to $N_S[h_S]$ implies that $(g_S,h_S)\in N_S\subseteq N_{\overline{S}}$. Hence $N_{\overline{S}}[g]=N_{\overline{S}}[g_S] = N_{\overline{S}}[h_S]=N_{\overline{S}}[h]$. So, $\theta$ is injective as well. Also for all $N_S[g]\in G/N_S$ there exists $N_{\overline{S}}[g]\in\widehat{G}/N_{\overline{S}}$ such that $\theta(N_{\overline{S}}[g])=N_S[g]$. So $\theta$ is surjective. Finally, $\theta(N_{\overline{S}}[g]N_{\overline{S}}[h])$ equals to $\theta(N_{\overline{S}}[gh])$ and that equals to $N_S[g_Sh_S]$ which is $N_S[g_S]N_S[h_S]$ and finally that equals to $\theta(N_{\overline{S}}[g])\theta(N_{\overline{S}}[h])$. So $\theta$ is an well defined group isomorphism, both algebraically and topologically. Hence $\widehat{G}/N_{\overline{S}}\cong G/N_S\cong\widehat{G}/\overline{N_S}$ which implies that $ \left|\widehat{G}/N_{\overline{S}}[1]\right|$ is equal to $\left|\widehat{G}/\overline{N_S}[1]\right|$. But since  $\overline{N_S}\subseteq N_{\overline{S}}$ one obtains $\overline{N_S}[1]\leq N_{\overline{S}}[1]\leq\widehat{G}$ and thus $\left|\widehat{G}/N_{\overline{S}}[1]\right|\left|N_{\overline{S}}[1]:\overline{N_S}[1]\right|$ equals to $\left|\widehat{G}/\overline{N_S}[1]\right|$. Hence $\left|N_{\overline{S}}[1]:\overline{N_S}[1]\right|=1$ which implies that $ N_{\overline{S}}[1]=\overline{N_S}[1]$ and thus $ N_{\overline{S}}=\overline{N_S}$ as each of them are congruences.
Thus our claim.
\end{proof}

\bibliographystyle{amsplain}

\begin{thebibliography}{1}

\bibitem{nB66}
N.~Bourbaki, \emph{General {T}opology}, Elements of {M}athematics,
  Addison-Wesley, Reading Mass., 1966.

\bibitem{bH77}
B.~Hartley, \emph{Profinite and residually finite groups}, Rocky Mountain J.
  Math \textbf{7} (1977), 193--217.

\bibitem{jK55}
J.~Kelley, \emph{General {T}opology}, D. Van Nostrand Company, Inc.,
  Toronto-New York-London, 1955.

\bibitem{jS80}
J.-P. Serre, \emph{{T}rees}, Springer Verlag, Berlin, Heidelberg, New York,
  1980.

\bibitem{jS83}
J.~R. Stallings, \emph{Topology of finite graphs}, Invent. Math. \textbf{71}
  (1983), 551--565.

\bibitem{jW98}
J.~Wilson, \emph{{P}rofinite {G}roups}, Oxford University Press, Oxford, 1998.

\bibitem{dD1}
B.~Das, \emph{{C}ofinite {G}raphs and {T}heir {P}rofinite {C}ompletions}, proquest, acumen.lib.ua.edu, 2013.

\bibitem{pZ89}
P.~A. Zalesskii, \emph{Geometric characterization of free constructions of profinite groups}, Siberian Math J. \textbf{30} (1989), 227--235.

\end{thebibliography}

\end{document}